\documentclass{article}
\usepackage{amsmath,amsfonts,amssymb,amsthm,tikz-cd,url}
\usepackage[left=2.5cm,right=2.5cm, top=3cm,bottom=3cm]{geometry}

\title{Binomial rings, and integral homology of complements of compact toric arrangements}
\author{Alexey G.~Gorinov, Alexander V.~Zakharov}
\date{}
\newtheorem{prop}[equation]{Proposition}
\newtheorem{lemma}{Lemma}
\newtheorem{theorem}{Theorem}
\newtheorem{cor}{Corollary}

\theoremstyle{remark}
\newtheorem{remark}{Remark}

\newcommand{\Q}{{\mathbb Q}}
\newcommand{\Z}{{\mathbb Z}}

\newcommand{\CC}{{\mathbb C}}

\newcommand{\R}{{\mathbb R}}
\newcommand{\Ab}{{\mathrm{Ab}}}
\newcommand{\DMO}{\DeclareMathOperator}
\DMO{\Gr}{Gr}
\DMO{\ar}{\longrightarrow}
\DMO{\al}{\longleftarrow}
\DMO{\ev}{ev}
\DMO{\rk}{rk}
\DMO{\pt}{pt}
\DMO{\Sym}{Sym}
\newcommand{\mb}[1]{\mathbb{#1}}
\newcommand{\os}[2]{\overset{#1}{#2}}
\DMO{\mmod}{\mbox{-}\mathrm{mod}}
\DeclareMathOperator{\Bin}{\mathrm{Bin}}
\DeclareMathOperator{\MV}{\mathrm{MV}}

\DeclareMathOperator{\colim}{\mathrm{colim}}

\DMO{\zgqis}{\overset{\sim}{\leftrightsquigarrow}}
\DMO{\Ch}{Ch}
\DMO{\Ext}{Ext}
\DMO{\Sets}{Sets}
\DMO{\op}{op}
\DMO{\Spaces}{Spaces}
\DMO{\Groups}{Groups}
\DMO{\Fin}{Fin}

\usepackage{hyperref}

\begin{document}

\maketitle

\begin{abstract}
An \emph{affine subtorus} of the compact torus $T=(S^1)^n$ is a translated copy of a Lie subgroup. Given a finite collection $T_1,\ldots, T_k$ of such subtori, and a prime $p$, we describe an explicit chain complex that calculates the group $H_*(T-\bigcup_{i=1}^k T_i,\mathbb{Z}_{(p)})$. 
Our main tool is the binomial models for spaces constructed by T.~Ekedahl. We use these results to express the groups $H_*(T-\bigcup_{i=1}^k T_i,\mathbb{Z})$. We also show that the Mayer-Vietoris spectral sequence that converges to the homology of $T-\bigcup_{i=1}^k T_i$ collapses at the second page rationally, and also integrally under some assumptions on the arrangement $T_1,\ldots, T_k$, with all extension problems being trivial in the latter case.
\end{abstract}

\section{Introduction}

There are quite a lot of results in the literature about the topology of the complement $T-\bigcup_{i=1}^k T_i$ where $T$ is the complex algebraic torus $(\CC^*)^n$, and $T_1,\ldots, T_k$ are (possibly translated) complex algebraic subgroups. 

For example, for any arrangement of subtori $T_i$, L.~Moci and R.~Pagaria in \cite{moci_pagaria}
described the homotopy type of the one point compactification of $\cup_i T_i$.
Moreover, there is an automorphism of $T$ that takes each $T_i$ to a subgroup which is closed in $\CC^n\supset (\CC^*)^n=T$. This observation allows one to apply the Alexander duality to recover the additive structure of the integral cohomology of $T-\cup T_i$. As a corollary, the authors deduce the degeneration of the integral Leray spectral sequence of the open embedding $T-\cup_i T_i\subset T$ at $E_2$. This gives one the multiplicative structure of $H^*(T-\cup_i T_i)$ up to passing to the associated graded ring.
There are more precise results on the rational homotopy type of $T-\cup_i T_i$ by C.~De Concini and G.~Gaiffi \cite{dcg1},
and on the integral cohomology in the case of a codimension one arrangement by F.~Callegaro and E.~Delucchi~\cite{cd}, and by F.~Callegaro, M.~D’Adderio, E.~Delucchi, L.~Migliorini, and R.~Pagaria~\cite{cdadmp}. It is unclear to us though if these results allow one to compute $H^*(T,\cup_i T_i)$.

Suppose however that $T$ is the compact torus $T=(S^1)^n$, and $T_1,\ldots, T_k\subset T$ is a finite collection of (possibly translated) Lie subgroups. This is the notation we will use in the rest of the paper. If $T$ and all $T_i$ are complex algebraic (and each $T_i\subset T$ is a closed embedding of complex algebraic varieties), then one can give a model for the rational homotopy type of $T-\bigcup_{i=1}^k T_i$, see A.~Zakharov~\cite{sasha_arxiv}. Also, the rational Betti numbers were calculated in some cases by P.~Kalugin in~\cite{kalugin}. This completes the list of the known results that we are aware~of.

Our goal in this note is to calculate the homology groups of $T^n-\bigcup_{i=1}^k T_i$ with coefficients in $\Z_{(p)}$ and~$\Q$, see Theorems~\ref{main},~\ref{main1},~\ref{rational}, and~\ref{main2}. The rational case (Theorem~\ref{rational}) is a practice example: it uses only the classical Mayer-Vietoris spectral sequence, but it also shows one what to look for when the coefficients are more complicated. To handle the $\Z_{(p)}$ case (Theorem~\ref{main}) we use cosimplicial binomial algebras. We do not know how to solve the problem using more standard methods.

For general results on modelling homotopy types by binomial algebras, and the relation to Grothendieck's schematisation programme we refer the reader to the papers~\cite{ekedahl} by T.~Ekedahl and~\cite{toen} by B.~To\"en, as well as to later work by G.~Horel~\cite{horel} and B.~Antieau~\cite{antieau}. As an illustration, the singular cochains of Abelian Eilenberg-Maclane spaces are free derived binomial algebras.
Moreover, the cochains functor taking values in derived binomial algebras provides a fully faithful embedding of the homotopy category of nilpotent spaces.
It is worthwhile to compare this with a well-known result by M.~Mandell, who showed in~\cite{mandell} that the cochain functor with values in $E_\infty$-algebras (and even derived commutative algebras) over the integers provides only a faithful embedding. 
The case of a nilpotent $K(\pi,1)$ was studied in~\cite{suciu} by R.~Porter and A.~Suciu. In this note we apply the binomial model of the simplest possible Eilenberg-MacLane space, namely the torus, to complements of arrangements.


Theorem~\ref{main1} provides a somewhat simpler answer than Theorem~\ref{main} for homology with coefficients in a general binomial ring $R$ (e.g.\ in $R=\Z$) assuming all intersections of the $T_i$'s are path connected. Assume in addition that the functor $I\mapsto H_1(\bigcap_{i\in I} T_i,\Z)$ from the subsets of $\{1,\ldots, k\}$ to free Abelian groups lifts to a functor to free groups. We prove then (Theorem~\ref{main2}) that the first page of the Mayer-Vietoris spectral sequence already gives one the answer, same as in the rational case (Theorem~\ref{rational}).

The cochain complexes given by Theorems~\ref{main} and~\ref{main1} have infinite rank, but it is easy enough to find appropriate truncations which result in finite rank complexes with the same homology, see Corollaries~\ref{finite_rank_thm1} and~\ref{finite_rank_thm2} respectively. Combining Corollary~\ref{finite_rank_thm1} (or Theorem~\ref{main}) with a bound on the set of the torsion primes (Proposition~\ref{bound_on_torsion}) one gets the integral homology of $T^n-\bigcup_{i=1}^k T_i$ for an arbitrary arrangement of affine subtori of $T$.

We would also like to mention that the homology groups of the hull of a quasi-periodic structure in the theory of quasi-crystals studied by J.~Bellissard~\cite{bellissard} and P.~Kalugin~\cite{kalugin} are isomorphic to those of $T^n-\bigcup_{i=1}^k T_i$ for some choice of $T$ and $T_1,\ldots, T_k$. This fact was the initial motivation behind this project.

\smallskip

{\bf Organisation of the note.} In Section~\ref{results} we state our main results, which we then prove in Sections~\ref{proofs} and~\ref{more_proofs}.

\smallskip

{\bf Acknowledgements.} We are grateful to Pyotr Pushkar and to Pavel Kalugin for informing us about the problem, respectively for explaining its applications to physics. The second named author is grateful to Dmitry Kubrak, Pierre Godfard and Grigory Solomadin for stimulating discussions about binomial algebras and toric diagrams. 

\section{Main results and some comments}\label{results}

\subsection{Notation and main results}

We will now review the ingredients needed in our main result. 

{\bf Some notation and coventions.} We let $\triangle$ denote the \emph{simplex category} (the objects are the sets $[n]=\{0,\ldots,n\}, n\in\Z_{\geq 0}$, and the morphisms are the order preserving maps). Also, let $\square_k$ be the \emph{$k$-cube category} (the objects are the subsets of $\{1,\ldots, k\}$, and the morphisms are the inclusions). If $\mathcal{C}$ is a category, functors $\triangle^{op}\ar\mathcal{C}$, respectively $\triangle\ar\mathcal{C},\square_k^{op}\ar\mathcal{C}$, and $\square_k\ar\mathcal{C}$ are called \emph{simplicial}, respectively \emph{cosimplicial}, \emph{cubical}, and \emph{cocubical} objects of $\mathcal{C}$. If $\mathcal{C},\mathcal{D}$ are categories, we denote the category of functors $\mathcal{C}\ar\mathcal{D}$ by $[\mathcal{C},\mathcal{D}]$. 

All rings below will be assumed associative and with identity. All ring maps are assumed to take $1$ to $1$. Let $A$ be a commutative ring. We denote the category of $A$-modules by $A\mbox{-}\mathrm{mod}$. We will abbreviate $\Z\mbox{-}\mathrm{mod}$ to $\Ab$. We will write $C^{\geq 0}(A)$ to denote the category of non-negatively graded cochain complexes of $A$-modules.

If $H$ is an Abelian group, we set $H[i]$ to be $H$ placed in degree $-i$ and viewed as a cochain complex, and we will sometimes write $H$ instead of $H[0]$.

Suppose $X_1,\ldots, X_k\subset T$ is a family of subspaces. For a non-empty $I\subset\{1,\ldots, k\}$ we set $X_I:=\bigcap_{i\in I} X_i$, and we define $X_\varnothing:=T$.

\smallskip

{\bf Binomial monad.} Recall that a commutative ring $R$ is called \emph{binomial} if the underlying Abelian group is torsion free, and for every $x\in R$ and $l\in\Z_{>0}$, the binomial coefficient $\binom{x}{l}:=\frac{x(x-1)\cdot\cdots\cdot (x-l+1)}{l!}\in R\otimes\Q$ in fact belongs to $R$. Examples of binomial rings include $\Z$ localised at any (possibly empty) set of primes, and any $\Q$-algebra. Binomial rings (and arbitrary ring maps) form a bicomplete category. The forgetful functor from binomial rings to Abelian groups has a left adjoint, which we will denote $\Bin$. Viewed as a functor $\Ab\ar \Ab$, $\Bin$ is a monad.

The free binomial ring $\Bin(\Z^m)$ on $\Z^m$ is the ring of the polynomials in $\Q[t_1,\ldots,t_m]$ that take integer values at integer points; as a ring, $\Bin(\Z^m)$ is generated by the binomial coefficients $\binom{t_i}{l}=\frac{t_i(t_i-1)\cdot\cdots\cdot (t_i-l+1)}{l!}$ for $i\in\{1,\ldots,m\}$ and $l\in\Z_{\geq 1}$. We extend $\Bin$ object-wise to a monad $$\Bin:[\triangle,\Ab]\ar [\triangle,\Ab]$$ on the category $[\triangle,\Ab]$ of cosimplicial Abelian groups, which is equivalent to $C^{\geq 0}(\Z)$ by the Dold-Kan correspondence, giving us a monad $$\Bin:C^{\geq 0}(\Z)\ar C^{\geq 0}(\Z).$$ We will use these two versions of $\Bin$ interchangeably.

Slightly generalising the above, suppose $A$ is a binomial ring. A \emph{binomial algebra} over $A$, or a \emph{binomial $A$-algebra} is a binomial ring $R$ equipped with a map of rings $A\ar R$. The forgetful functor from binomial $A$-algebras to $A$-modules has a left adjoint, denoted $\Bin_A(-)$, which is a monad when viewed as an endofunctor of $A\mbox{-}\mathrm{mod}$. We have for example $\Bin_A(A^m)=\Bin(\Z^m)\otimes_\Z A$. Same as for $A=\Z$, the category of binomial $A$-algebras is bicomplete. Repeating verbatim the above construction we get a monad, which we denote $\Bin_A$, on the category of cosimplicial $A$-modules, or equivalently, on $C^{\geq 0}(A)$.

We refer the reader to~\cite{ekedahl,elliot, horel, ksz} for more information about binomial algebras.

\smallskip

{\bf A cochain complex $A^*$.} 
\label{construction_A}
Let us identify $T=\R^n/\Z^n$, and let $pr:\R^n\ar T$ be the projection. For a \emph{connected} affine subtorus $T'\subset T$ and a subring $R\subset\R$ we will now define a cochain subcomplex $A^*(T',R)$ of $R$-modules of the de Rham algebra $\Omega^*(U)$ where $U$ is a tubular neighbourhood of $T'$.  We will write $\langle m_1,\ldots, m_l\rangle$ to denote the $R$-submodule of an $R$-module $M$ generated by $m_1,\ldots, m_l\in M$. 

The complex $A^*(T',R)$ will be concentrated in degrees $0$ and $1$. The degree $1$ component $A^1(T',R)$ is set to be $\langle dx_1|_U,\ldots, dx_n|_U\rangle$ where $x_1,\ldots,x_n$ are the standard coordinates in $\R^n$. 
The degree $0$ component $A^0(T',R)$ is the $R$-module 
generated by all continuous functions $f:U\ar \R$ such that after lifting $f$ to $pr^{-1}(U)\subset\R^n$, on each path component we have
\begin{equation}\label{eq:affine} 
f=\sum_{i=1}^n a_i x_i + a_0
\end{equation} for some $a_0,a_1,\ldots,a_n\in R$ and $f$ is constant along the subspace over $T'$ (here $a_0$ may depend on the component).

Observe that $A^1(T',R)$ is canonically identified with $\langle dx_1,\ldots,dx_n\rangle$, while $A^0(T',R)$ is non-canonically isomorphic to $\langle 1,f_1,\ldots, f_{n-\dim T'}\rangle$ where $f_1,\ldots, f_{n-\dim T'}:\R^n\ar \R$ are linear functions with integer coefficients that define the tangent space 
$T_{T'}\subset \R^n$. The latter makes sense since the vector subspace $T_{T'}$ is automatically defined over $\Z$: the kernel of the exponential map from $T_{T'}$ to $T'$ is a natural sublattice. 
The element $1\in A^0(T',R)$ is well defined, and by the connectivity of $T'$ its class spans $H^0(A^*(T',R))$.

\begin{remark}\label{rmrk:B_construction}
Note that along with $A^*(T',R)$ one can consider the functors $\tau_{\le 1}C^*_{\mathrm{sing}}(T',R)$ and $H^0(T',R)\oplus H^1(T',R)[-1]$. The latter might be viewed as the ``associated graded'' version of the former with respect to the canonical filtration.
In the sequel we will see that $A^*(T',R)$ is a manageable model of $\tau_{\le 1}C^*_{\mathrm{sing}}(T',R)$.
\end{remark}
\smallskip

{\bf A binomial algebra.} Let us now assume that $R$ is binomial. Let $A^*\in C^{\geq 0}(R)$, and suppose we have a \emph{coaugmentation} of $A^*$, i.e.\ a cochain map $R\ar A^*$ where $R$ on the left lives in degree 0. We set $B(A^*)$ to be the pushout of (cosimplicial) binomial algebras in the following diagram:
\begin{equation}\label{eq:bin_pushout}
\begin{tikzcd}
\Bin R \arrow[dr, phantom, "\lrcorner", very near end]\ar[r]\ar[d] & R\ar[d]\\
\Bin(A^*)\ar[r] & B(A^*).
\end{tikzcd}
\end{equation}
Here $\Bin R$ in the top left corner is $\Bin$ applied to $R$ in degree $0$. The left vertical arrow is $\Bin$ of the coaugnemtnation $R\ar A^*$. 
The top horizontal arrow is the Bin-algebra structure on $R$, given by evaluating $R$-valued polynomials at $1$.
Concretely, $B(A^*)=\Bin(A^*)\otimes_{\Bin(R)}R$, where the tensor product is the object-wise tensor product in the category of cosimplicial modules over the (constant) cosimplicial ring $\Bin(R)$. 

Observe that for a connected affine subtorus $T'\subset T$ we have a natural coaugmentation $R\ar A^*(T',R)$ which takes $1\in R$ to $1\in A^0(T',R)$, so the above construction gives us a binomial algebra $B(A^*(T',R))$. For $I\subset\{1,\ldots, k\}$ we set $B(T_I,R)$ to be the direct sum of $B(A^*(T',R))$ for all path components $T'\subset T_I$.

\smallskip

{\bf A few cocubical cochain complexes.}
Let $R$ be a binomial ring. We then have the following functors $\square_k\ar C^{\ge 0}(R)$, i.e. \emph{cocubical cochain complexes} of $R$-modules:
$$I\mapsto B(T_I,R),\ I\mapsto \Bin (H^1(T_I,R)[-1]).$$ 
(We will use the second of these functors only when all $T_I$ are path connected.)

Also, for an arbitrary commutative ring $R$ we have the functor $\square_k\ar C^{\ge 0}(R)$ given by
$$I\mapsto H^*(T_I,R),$$ where we regard the right hand side as a cochain complex with zero differential.

\smallskip

{\bf Mayer-Vietoris construction.} Recall that starting from a cocubical cochain complex $\{C_I^*\}_{I\subset\{1,\ldots, k\}}$, one can construct a double complex as follows. Set $$C^{p,q}:=\bigoplus_{|I|=p} C^q_I.$$ We have the differential $d:C^{p,q}\ar C^{p,q+1}$ induced by the differentials of the individual complexes $C^*_I$. We also have the \emph{combinatorial differential} $\delta:C^{p,q}\ar C^{p+1,q}$ defined by
$$(\delta c)_{I=\{i_1<\cdots < i_{p+1}\}}:=\sum_{j=1}^{p+1} (-1)^{j+1}c_{I\setminus\{i_j\}}|_I.$$ Here $c=(c_J)\in C^{p,q}$ and $\delta c=((\delta c)_I)\in C^{p+1,q}$ where $I,J\subset\{1,\ldots, k\}, |J|=p$, and $|I|=p+1$. The restriction map $C_J^q\ar C_I^q$ in the right hand side is induced by the inclusion $J\subset I$.

The differentials $d$ and $\delta$ commute. This turns $(C^{p,q})$ into a double complex, which we will call the \emph{Mayer-Vietoris double complex} of $\{C_I^*\}$ and denote $\MV(\{C_I^*\})$. We will use the same notation for the resulting total complex.


\smallskip

{\bf Results.} We can now state our main results.

\begin{theorem}\label{main}
Let $p$ be a prime. The complex $\MV(\{B(T_I,\Z_{(p)})\})$ computes the cohomology
$$H^*(T,\bigcup_{i=1}^k T_i,\Z_{(p)})\cong H_{\dim T-*}(T-\bigcup_{i=1}^k T_i,\Z_{(p)}).$$
\end{theorem}

\smallskip

If all intersections of the affine subtori that we wish to delete are path connected, then the answer simplifies:

\begin{theorem}\label{main1}
Suppose all non empty intersections $T_I,I\subset\{1,\ldots,k\}$ are path connected, and let $R$ be a binomial ring. The complex $\MV(\{\Bin (H^1(T_I,R)[-1])\})$ computes the cohomology
$$H^*(T,\bigcup_{i=1}^k T_i,R)\cong H_{\dim T-*}(T-\bigcup_{i=1}^k T_i,R).$$
\end{theorem}
\smallskip

Everything simplifies even further if we are instead interested in homology with coefficients in a $\Q$-algebra:

\begin{theorem}\label{rational}
Let $R$ be a $\Q$-algebra. The complex $\MV(\{H^*(T_I,R)\})$ computes the cohomology
$$H^*(T,\bigcup_{i=1}^k T_i,R)\cong H_{\dim T-*}(T-\bigcup_{i=1}^k T_i,R).$$
\end{theorem}

In a sense, Theorem~\ref{rational} gives one the smallest model one could hope for. We do not know if the conclusion of the theorem remains true if $R$ is no longer assumed a $\Q$-algebra, but here is a partial result in this direction. 

\begin{theorem}\label{main2}
Assume that every intersection $T_I, I\subset\{1,\ldots, k\}$ is path connected, and that the functor $I\mapsto H_1(T_I;\mb{Z})$ from $\square_k$ to $\Ab$ lifts to a functor from $\square_k$ to free groups.

Then the complex $\MV(\{H^*(T_I;\mb{Z})\})$ computes the cohomology $$H^*(T,\bigcup_{i=1}^k T_i,\Z)\cong H_{\dim T-*}(T-\bigcup_{i=1}^k T_i,\Z).$$ 
\end{theorem}

\begin{remark}
In fact, one can prove a more general result. The connected components of all intersections $T_I$ form a poset $\mathcal{C}$ which refines $\square_k$ via a natural poset map $\mathcal{C}\ar \square_k$. Assume again that the functor $H_1(-,\Z)$ from $\mathcal{C}$ to Abelian groups lifts to free groups. Note that the construction $A^*(-,\Z)$ can be viewed as a diagram of 2-term complexes of shape $\mathcal{C}$. Assume that this diagram is quasi-isomorphic to $H^*(A^*(-,\Z))$ via a zig-zag of quasi-isomorphisms that induces the identity on the cohomology. (As we will see in Lemma~\ref{lemma:decomposition}, the latter hypothesis is automatic when all $T_I$ are assumed path connected.) 

Then the complex $\MV(\{H^*(T_I;\mb{Z})\})$ again computes $H_{\dim T-*}(T-\bigcup_{i=1}^k T_i,\Z)$.
\end{remark}

We prove Theorems~\ref{main}, \ref{main1} and~\ref{rational} in Section~\ref{proofs}, and Theorem~\ref{main2} in Section~\ref{more_proofs}.

\subsection{Comments}\label{sec_comm}
%

To begin with, let us see why Theorems~\ref{main} and~\ref{main1} in fact give us a finite recipe for calculating the homology groups of $T-\bigcup_{i=1}^k T_i$. (In the case of Theorems~\ref{rational} and~\ref{main2} this is clear.)

{\bf Bounding the torsion.} Theorem~\ref{main} only gives one a $p$-local answer. Let us find a set of primes which contains all primes $p$ such that $H_*(T-\bigcup_{i=1}^k T_i,\Z)$ has $p$-torsion. This can be done using the next proposition.

\begin{prop}\label{bound_on_torsion}
Let $p$ be a prime. The cardinality of the $p$-torsion subgroup of $H_*(T-\bigcup_{i=1}^k T_i,\Z)$ divides that of
$\MV(\{H^*(T_I^*,\Z)\})$.
\end{prop}

We will prove this proposition in Section~\ref{proofs}.
\smallskip

{\bf Finite rank models.} Next, observe that if $T'\subset T$ is an affine subtorus and $R$ is a binomial ring, the cochain complexes $B(T',\Z_{(p)})$ and $\Bin( H^1(T',R))$ have infinite rank, one of the reasons being that $\Bin(\Z^m)$ has infinite rank as an Abelian group if $m>0$. So we want to replace each $B(T',\Z_{(p)})$ and $\Bin( H^1(T',R))$ with a finite rank cochain complex, functorially in $T'$.

The monad $\Bin:\Ab\ar\Ab$ has a natural \emph{binomial} filtration $(\Bin_{\leq i})_{i\in\Z_{\geq 0}}$ by subfunctors. For example, $\Bin(\Z^m)_{\leq i}$ is the additive group of all polynomials $\in\Q[t_1,\ldots, t_m]$ which take integer values at integer points and have degree $\leq i$; all groups $\Bin(\Z^m)_{\leq i}$ are free Abelian of finite rank. 

Tensoring everything by a binomial ring $R$ we get a filtration $(\Bin_{R,\leq i})_{i\in\Z_{\geq 0}}$ on $\Bin_{R}$. We extend this filtration object-wise to $\Bin_{R}:[\triangle,R\mbox{-}\mathrm{mod}]\ar [\triangle,R\mbox{-}\mathrm{mod}]$, and hence to $\Bin_{R}:C^{\geq 0}(R)\ar C^{\geq 0}(R)$.

Let $\tau_{\leq i}$ denote the canonical filtration functor on cochain complexes. 
\begin{lemma}\label{lemma:bin_vs_canonical}
Let $R$ be a binomial ring, and let $A$ be a finitely generated free $R$-module.
Then the canonical and binomial filtrations on $\Bin_R(A[-1])$ are equivalent:
for any $N\ge m$ the maps
$$\tau_{\leq m}\Bin_{R,\leq N} (A[-1])\al \tau_{\leq m}\Bin_{R,\le m}(A[-1]) \ar \Bin_{R,\le m}(A[-1])$$
are quasi-isomorphisms.
\end{lemma}
\begin{proof}
        The functor $\Bin_{R,\le m}(-)/\Bin_{R,\le m-1}(-)$ is naturally isomorphic to the $m$-th divided powers functor $\Gamma^n(-)$, see~\cite[Remark 2.8]{ksz}. So the binomial filtration induces a spectral sequence
        $$E^{pq}_1=H^{p+q}(\Gamma^{-p}(A[-1]))\implies H^{p+q}(\Bin_R(A[-1])).$$
        We have the d\'ecalage isomorphism $\Gamma^{-p}(A[-1])\cong \Lambda^{-p}(A)[p]$ where $\Lambda^*(-)$ is the \emph{Grassmann} (or \emph{exterior}) algebra functor, see \cite[(4.3.5)]{Illusie}.
        So we have $E^{pq}_1=H^{2p+q}(\Lambda^{-p}A)$, which implies that the terms $E^{pq}_1$ vanish outside the line $2p+q=0$, and
        hence the spectral sequence degenerates at $E_1$. 
        Thus the map
        $$H^*(\Bin_{R,\le -p}(A[-1]))\ar H^*(\Bin_R(A[-1]))$$ is isomorphic to the inclusion 
        of $\Lambda^{\le -p}(A)$ in the Grassmann algebra $\Lambda^*(A)$.
        This implies that the binomial and canonical filtrations are equivalent.
        
        
\end{proof}

This allows one to construct a finite rank model for the cochain complex given by Theorem~\ref{main1}:

\begin{cor}\label{finite_rank_thm2}
We use the notation and assumptions of Theorem~\ref{main1}. By setting 
\begin{equation}\label{def_tilde_hbin}
I\mapsto\tau_{\leq \dim T}\Bin_{\leq \dim T}(H^1(T_I,R)[-1]),I\subset \{1,\ldots,k\}
\end{equation}
we get a cocubical cochain complex such that the sum of the ranks of all cochain and cubical components is finite. The total cochain complex of the corresponding Mayer-Vietoris double complex calculates
$$H^*(T,\bigcup_{i=1}^k T_i,R)\cong H_{\dim T-*}(T-\bigcup_{i=1}^k T_i,R).$$
\end{cor}

$\square$

\begin{remark}\label{remark:why_filtration_so_high_1}
Note that in~(\ref{def_tilde_hbin}) we take the $\dim T$-th terms of both the canonical and the binomial filtrations for all $I$, even though for each individual $T_I$ the $\dim T_I$-th terms would suffice by Lemma~\ref{lemma:bin_vs_canonical}. The reason is that we need the result to be a cocubical cochain subcomplex of $\{\Bin (H^1(T_I,R)[-1])\}$.
\end{remark}

\smallskip

Now we generalise this to the non-connected case. 
Consider a coaugmented two term complex $A^*=[A^0\ar A^1]$ of finitely generated $R$-modules. Denote the coaugmentation by $i:R\ar A^*$, and assume that $H^1(A^*)$ is a free $R$-module and $i$ induces an isomorphism of $H^0$.
We have the pushout $B(A^*)$ given by diagram~\eqref{eq:bin_pushout}.
\begin{lemma}\label{lemma:bin_vs_canonical_2}
The pushout $B(A^*)$ admits a natural filtration $B_{\le}(A^*)$ with properties similar to those in Lemma~\ref{lemma:bin_vs_canonical}, namely for any $N\ge m$ the maps
$$\tau_{\leq m}B_{\le N}(A^*)\al \tau_{\leq m}B_{\le m}(A^*) \ar B_{\le m}(A^*)$$
are quasi-isomorphisms.
\end{lemma}
\begin{proof}
The coaugmentation map $i:R\ar A^*$ turns $\Bin(A^*)$ into a (right) $\Bin(R)$-module.
For a given $N$
we consider the product $\Bin_{\le N}(A^*)\cdot \Bin(R)$ 
as a subcomplex of $\Bin(A^*)$ via $i$.  
Set $$B_{\le N}(A^*)= 
(\Bin_{\le N}(A^*)\cdot \Bin(R))\otimes_{\Bin(R)}R$$
where the tensor product is again taken in the category of cosimplicial modules over $\Bin (R)$. The result coincides
with its image in $B(A^*)$, cf.\ the remarks after diagram~(\ref{eq:bin_pushout}).
Note that for $R=\Z_{(p)}$ the canonical filtration on
$A^*$ splits (non-canonically), providing a filtered quasi-isomorphism $B(A)\ar \Bin(H^1(A^*)[-1])$. It remains to apply Lemma~\ref{lemma:bin_vs_canonical}.
\end{proof}


\begin{cor}\label{finite_rank_thm1}
We use the notation of Theorem~\ref{main}. By setting 
\begin{equation}\label{def_tilde_bin}
I\mapsto\tau_{\leq \dim T}B_{\leq \dim T}(T_I,R), I\subset\{1,\ldots,k\}
\end{equation}
we get a cocubical cochain complex such that the sum of the ranks of all cochain and cubical components is finite. The total cochain complex of the corresponding Mayer-Vietoris double complex calculates
$$H^*(T,\bigcup_{i=1}^k T_i,R)\cong H_{\dim T-*}(T-\bigcup_{i=1}^k T_i,R).$$
\end{cor}

$\square$

\begin{remark}\label{remark:why_filtration_so_high_2}
Note that in~(\ref{def_tilde_bin}) we again take the $\dim T$-th terms of both the canonical and the binomial filtrations for all $I$, cf.\ Remark~\ref{remark:why_filtration_so_high_1}.
\end{remark}

\bigskip

{\bf Size of the models.} We do not know if Theorem~\ref{rational} remains true if one replaces $\Q$ by $\Z$ in it. If it does, then we get a smaller cochain complex that calculates the homology of $T-\bigcup_{i=1}^k T_i$ than the ones given by Corollaries~\ref{finite_rank_thm1} and~\ref{finite_rank_thm2}. Let us quantify the word ``smaller''. Below by the \emph{size} of a cochain complex we mean the total rank. Let us take a component $T'$ of some $T_I$ and estimate the size of its contribution. Set $r=\dim T'$.

The contribution of $T'$ to the model of Theorem~\ref{rational} is simply $2^r$, the total rank of the cohomology of $T'$.

\smallskip

Let us compare this with the models given by Corollaries~\ref{finite_rank_thm2} and~\ref{finite_rank_thm1}. For our estimates below we use the non-normalised Moore complex of a cosimplicial group instead of the normalised version in the Dold-Kan correspondence, keeping in mind that the size of the former is greater than or equal the size of the latter.

Let $K^*=[K^0\ar K^1]$ be a $2$-term complex of free Abelian groups with $K^0$ in degree zero.
Recall that the (non-normalised) Moore complex of the cosimplicial group corresponding to $K^*$
in degree $m$ is equal to $K^0\oplus (K^1)^{\oplus m}$. Recall also that we have set $n=\dim T$. 
Let $$D(n,d):=\rk \Bin_{\le n}(\mb{Z}^d)=\rk \Sym^{\le n}(\mb{Z}^d)=
\rk \Sym^{n}(\mb{Z}^{d+1})=\binom{n+d}{n}.$$

The size of Moore complex corresponding to $\tau_{\le n}\Bin_{\le n}(H^1(T')[-1])$
is less than or equal
\begin{equation}\label{estimate1}
D(n,n)+D(n,2n)+\cdots+D(n,nr).
\end{equation}
This is the contribution of $T'$ to the model of Corollary~\ref{finite_rank_thm2}. It is $\leq $ the contribution of $T'=T$, which is $e^{(1+o(1))n\ln n}$ as $n\ar\infty$, and in general (\ref{estimate1}) is $\sim D(n,nr)=e^{(1+o(1))n\ln r}$ as $n\ar\infty$
under the assumption $1/r= o(1)$.

\smallskip

Similarly, the size of $\tau_{\le n}\Bin_{\le n}(K^*)$
is less than or equal than the sum
$$D(n,a+b)+D(n,a+2b)+\cdots+D(n,a+nb)$$
where $a=\rk K^0,b=\rk K^1$. This is the contribution of $T'$ to the model of Corollary~\ref{finite_rank_thm1} if we set $a=n-r+1$ and $b=n$. It is $\leq e^{(1+o(1))n\ln n}$ as $n\ar\infty$, uniformly in $r$.

\section{Proofs of Theorems~\ref{main}, \ref{main1} and~\ref{rational}}\label{proofs}

Let us start with some preliminary remarks which will be useful both in the rational and $\Z_{(p)}$ cases.

{\bf A spectral sequence.} If $\{C_I^*\}_{I\subset\{1,\ldots, k\}}$ is a cocubical cochain complex, there is a \emph{vertical filtration} on the double complex $(C^{p,q}):=\MV(\{C_I^*\})$: the $t$-th term of the filtration is the double complex with $(p,q)$-component equal $C^{p,q}$ for $p\geq t$, and $0$ otherwise. This filtration induces a spectral sequence with
\begin{equation}\label{spseq}
E^{p,q}_1=\bigoplus_{|I|=p} H^q(C^*_I)\implies H^*\bigl(\MV(\{C_I^*\})\bigr).
\end{equation}

\medskip

{\bf Mayer-Vietoris principle.} We will now describe standard cochain complexes that calculate the relative cohomology of $T$ modulo the union of a family of subspaces.

Let $R$ be a commutative ring, and suppose the torus $T$ has been triangulated so that each $T_i$ is a subpolyhedron. We use $C^*(-,R)$ to denote the simplicial cochains of the triangulation with coefficients in $R$. If $X_1,\ldots,X_k$ is a finite family of subpolyhedra, then $I\mapsto C^*(X_I,R)$ defines a cocubical cochain complex, as does $I\mapsto \Omega^*(T_I)$.

\begin{lemma}\label{standard_complexes}
The total complex of $\MV(\{C^*(X_I,R)\})$ calculates $H^*(T,\bigcup_{i=1}^k X_i,R)$, for an arbitrary commutative ring $R$. Moreover, the total complex of $\MV(\{\Omega^*(T_I)\})$ calculates $H^*(T,\bigcup_{i=1}^k T_i,\R)$.
\end{lemma}
$\square$

\medskip

We are now ready to prove Theorems~\ref{main},~\ref{main1}, and~\ref{rational}. We start with the latter.

{\bf Proof of Theorem~\ref{rational}.} It suffices to prove Theorem~\ref{rational} in the case $R=\R$, so let us assume this. If $T'\subset T$ is an affine subtorus, we have an inclusion $H^*(T',\R)\ar \Omega^*(T')$ given by taking each class to its translation invariant representative. Moreover, if $T''\subset T'$ is a smaller affine subtorus, the next diagram commutes:
$$
\begin{tikzcd}
H^*(T',\R)\ar[r]\ar[d] & \Omega^*(T')\ar[d]\\
H^*(T'',\R)\ar[r] & \Omega^*(T'').
\end{tikzcd}
$$

So we get a map $\{H^*(T_I,\R)\}\ar \{\Omega^*(T_I)\}$ of cocubical cochain complexes which is a quasi-isomosphism on every cubical component. Using spectral sequence~(\ref{spseq}) we see that the map of the total complexes of the Mayer-Vietoris constructions is also a quasi-isomorphism. Theorem~\ref{rational} now follows from Lemma~\ref{standard_complexes}.$\square$

\medskip

{\bf Auxiliary lemmas.} Here we prove a few lemmas which will be needed in the proof of Theorems~\ref{main} and~\ref{main1}. Let $R$ be a binomial ring.

Let $F:R\mbox{-}\mathrm{mod}\ar R\mbox{-}\mathrm{mod}$ be a functor. We extend $F$ object-wise to a functor $[\triangle,R\mbox{-}\mathrm{mod}]\ar [\triangle,R\mbox{-}\mathrm{mod}]$ and hence to a functor $C^*(R\mbox{-}\mathrm{mod})\ar C^*(R\mbox{-}\mathrm{mod})$ via the Dold-Kan equivalence. Abusing notation we will also denote by $F$ both these extensions. 

\begin{lemma}\label{left_quillen} 
The functor $F:C^*(R\mbox{-}\mathrm{mod})\ar C^*(R\mbox{-}\mathrm{mod})$ preserves quasi-isomorphisms between bounded cochain complexes of free non-negatively graded $R$-modules.
\end{lemma}
\begin{proof}
This is a standard fact due to Dold. 
Namely, $F:[\triangle,R\mbox{-}\mathrm{mod}]\ar [\triangle,R\mbox{-}\mathrm{mod}]$ preserves homotopy between maps of cosimplicial objects~\cite{Dold}.
On the other hand, cosimplicial homotopy maps to cochain homotopy under the Dold-Kan equivalence.
Thus $F\colon C^{\ge 0}(R\mmod)\ar C^{\ge 0}(R\mmod)$ preserves cochain homotopy.
It remains to note that a quasi-isomorphism between bounded complexes with projective terms is a homotopy equivalence.
\end{proof}

Let $A^*$ be a cochain complex of $R$-modules equipped with a coaugmentation $i:R\ar A^*$. A \emph{splitting} of $i$ is a cochain map $p:A^*\ar R$ such that $p\circ i$ is the identity of $R$.

\begin{lemma}\label{b_constr_qis}
Let $A_1^*\ar A_2^*$ be a quasi-isomorphism of bounded coaugmented non-negatively graded cochain complexes of free $R$-modules. Assume the coaugmentation of each of the complexes $A_1^*$ and $A^*_2$ has a splitting. Then the induced map $B(A_1^*)\ar B(A_2^*)$ of $\Bin_R$-algebras (see diagram~(\ref{eq:bin_pushout})) is also a quasi-isomorphism.
\end{lemma}
\begin{proof}
The induced map of the associated graded 
complexes with respect to the binomial filtration from Lemma~\ref{lemma:bin_vs_canonical_2}
is equal to $\Gamma(A^*_1/R)\ar \Gamma(A^*_2/R)$.
It is a quasi-isomorphism by Lemma~\ref{left_quillen} applied to the
quasi-isomorphism $A^*_1/R\ar A^*_2/R$, hence so is $B(A_1^*)\ar B(A_2^*)$.
\end{proof}

\begin{lemma}\label{splitting}
Let $A^*$ be a bounded coaugmented non-negatively graded cochain complex of free $R$-modules. Assume the coaugmentation of $A^*$ admits a splitting $p\colon A^*\ar R$. 

Then the induced composite map $\Bin(\ker p)\ar \Bin(A^*)\ar B(A^*)$ is a quasi-isomorphism.
\end{lemma}
\begin{proof}
Since $A^*\cong R[0]\oplus\ker p$, we have $B(A^*)\cong \Bin(\ker p)$.
\end{proof}
\begin{remark}\label{torsionfree}
It will be convenient to apply the above lemmas to the 
larger category of all torsion free $R$-modules.
Namely, for a torsion free group $A\in \Ab^{tf}$ one can let 
$\Bin(A)=\colim_{L\ar A}\Bin(L)$, where $L\ar A$ runs over all free Abelian groups and maps to $A$. 
It follows that, if $A^*\ar B^*$ is a quasi-isomorphism of complexes of torsion free Abelian groups, then 
$\Bin(A^*)\ar \Bin(B^*)$ is a quasi-isomorphism as well.
\end{remark}

Let $T'$ be a triangulated torus. Using the triangulation we can view $T'$ as a simplicial set. The $R$-valued functions on the simplices become a cosimplicial $R$-algebra, which is binomial as $R$ is. By the Dold-Kan correspondence we can view this as an algebra over the monad $\Bin_{R}$ on the category $C^{\geq 0}(R)$ with the underlying cochain complex being $C^*(T',R)$.

Using the natural coaugmentation $R\ar \tau_{\leq 1} C^*(T',R)$ we construct the algebra $B(\tau_{\le 1}C^*(T',R))$, see diagram~\ref{eq:bin_pushout} and Remark \ref{rmrk:B_construction}. We have a commutative square of $\Bin_{R}$-algebras in $C^{\geq 0}(R)$
\begin{equation}\label{tau_leq_1_to_c}
\begin{tikzcd}
        \Bin_R(R)\arrow[d]\arrow[r] & \Bin_R(\tau_{\le 1}C^*(T',R))\arrow[d]\\
        R\arrow[r] & C^*(T',R)
\end{tikzcd}
\end{equation}
where the left arrow $\Bin_R(R)\ar R$ is the structure map of the binomial algebra $\Bin_R R$. Diagram~(\ref{tau_leq_1_to_c}) gives us a map $B(\tau_{\le 1}C^*(T',R))\ar C^*(T',R)$ (again of $\Bin_R$-algebras in $C^{\geq 0}(R)$).
\begin{lemma}\label{lemma:bin_model_of_torus}
        The map
        $$B(\tau_{\le 1}C^*(T',R))\ar 
        C^*(T',R)$$
        is a quasi-isomorphism.
\end{lemma}
\begin{proof}
Any point $x\in T'$ defines a splitting $\ev_x\colon C^*(T',R)\ar R$ obtained by evaluating at $x$.
        Applying Lemma~\ref{splitting} we conclude that there is a quasi-isomorphism of binomial algebras
        $s_x\colon B(\tau_{\le 1}C^*(T',R))\cong \Bin_R(H^1(T',R)[-1])$.
        Note that $s_x$ is filtered with respect to the binomial filtration described in Lemma~\ref{lemma:bin_vs_canonical_2}.
        The proof of the lemma implies that the map 
        $H^*(B(H^1(T',R)[-1]))\ar H^*(T',R)$ is an isomorphism. 

\end{proof}
\begin{remark}
In fact free binomial algebras model cochains of Eilenberg-Maclane spaces.
For example, there is a natural quasi-isomorphism $\Bin(A[-n])\ar C^*_{\mathrm{sing}}(K(A^{\vee},n),\Z)$ for any $n>0$ and finitely generated free Abelian group $A$, see~\cite{ekedahl}.
\end{remark}

\medskip

{\bf Proof of Theorem~\ref{main}.} The proof of Theorem~\ref{rational} given above does not extend in a straightforward way to $\Z_{(p)}$ coefficients, but it gives us an idea of what to look for. Let us triangulate $T=\R^n/\Z^n$ so that
\begin{itemize}
\item the preimage in $\R^n$ of every simplex is a disjoint union of genuine affine simplices with all coordinates of each vertex being in $\Z_{(p)}$;
\item each subtorus $T_i$ has a closed polyhedral tubular neighbourhood $X_i$ such that for every non-empty $I\subset\{1,\ldots,k\}$ 
\begin{itemize} 
\item the inclusion $T_I\ar X_I$ is a homotopy equivalence, and 
\item for every path component $T'$ of $T_I$, the corresponding path component of $X_I$ fits into the tubular neighbourhood $U$ from the definition of $A(T',\Z_{(p)})$.
\end{itemize}
\end{itemize}

Let $T'$ be a path component of some $T_I,I\subset\{1,\ldots, k\}$, and let $X'$ be the corresponding path component of $X_I$. We have a cochain map 
\begin{equation}\label{compar}
A(T',\Z_{(p)})\ar C^*(X',\Z_{(p)})
\end{equation}
which induces a quasi-isomorphism $A(T',\Z_{(p)})\ar \tau_{\leq 1}C^*(X',\Z_{(p)})$.

We can view $C^*(X',\Z_{(p)}$ as an algebra over the monad $\Bin_{\Z_{(p)}}$ in $C^{\geq 0}(\Z_{(p)})$, cf.\ the remarks after Lemma~\ref{splitting} and before Lemma~\ref{lemma:bin_model_of_torus}. The map~(\ref{compar}) induces a map of $\Bin_{\Z_{(p)}}$-algebras
$\Bin(A(T',\Z_{(p)}))\ar C^*(X',\Z_{(p)})$,
which in turn induces a map of $\Bin_{\Z_{(p)}}$-algebras 
\begin{equation}\label{compar1}
B(T',\Z_{(p)})\ar C^*(X',\Z_{(p)}).
\end{equation}

This map factorises as $$B(T',\Z_{(p)})=B(A^*(T',\Z_{(p)}))\ar B(\tau_{\leq 1} C^*(T',\Z_{(p)}))\ar C^*(X',\Z_{(p)}).$$ The coaugmentation of $A^*(T',\Z_{(p)})$ has a splitting, as does the coaugmentation of $\tau_{\leq 1} C^*(T',\Z_{(p)})$. In the latter case the splitting is given by evaluating at any point $x\in T'$. So we can use Lemmas~\ref{b_constr_qis} and~\ref{lemma:bin_model_of_torus} to conclude that~(\ref{compar1}) is a quasi-isomorphism. The maps~(\ref{compar1}) induce a map of cocubical cochain complexes
$$\{B(T_I,\Z_{(p)})\}\ar \{C^*(T_I,\Z_{(p)})\}$$ which we have just seen to be a quasi-isomorphism on every cubical component. We now complete the proof of Theorem~\ref{main} by applying spectral sequence~(\ref{spseq}) and Lemma~\ref{standard_complexes}.$\square$

\medskip

{\bf Proof of Theorem~\ref{main1}.} Let us take an arbitrary triangulation of $T$ such that each $T_i$ is a subpolyhedron. We let $C^*(-,R)$ denote the corresponding simplicial cochain functor. Since $T_{\{1,\ldots, k\}}$ is assumed path connected, in particular non-empty, we can choose an $x\in T_{\{1,\ldots, k\}}$. Let $I\subset\{0,\ldots, k\}$. We set $\ev_x:C^*(T_I,R)\ar R$ be the evaluation map at $x$, which splits the coaugmentation $R\ar C^*(T_I,R)$. We have the following arrows:
\begin{equation}\label{factorisation_2nd_proof}
\Bin_R(H^1(T_I,R)[-1])\gets\Bin_R(\ker\ev_x)\ar\Bin_R(\tau_{\leq 1} C^*(T_I,R))\ar B(\tau_{\leq 1} C^*(T_I,R))\ar C^*(T_I,R).
\end{equation}
Here the arrow on the left is induced by the cochain map $\ker\ev_x\ar H^1(T_I,R)[-1]$ and is a quasi-isomorphism by Lemma~\ref{left_quillen}. The composition of the next two arrows is a quasi-isomorphism by Lemma~\ref{splitting}, and the last arrow is a quasi-isomorphism by Lemma~\ref{lemma:bin_model_of_torus}. The maps~(\ref{factorisation_2nd_proof}) induce maps of the corresponding cocubical cochain complexes, and we finish the proof using spectral sequence~(\ref{spseq}) and Lemma~\ref{standard_complexes} again.$\square$

\medskip

{\bf Proof of Proposition~\ref{bound_on_torsion}.} Triangulate $T$ so that every $T_i$ becomes a subpolyhedron. Let $(E^{p,q}_r,d_r)$ be the spectral sequence~(\ref{spseq}) for the cocubical cochain complex $\{C^*(T_I,\Z)\}$ constructed using the family $T_1,\ldots T_k$ of subpolyhedra. Observe that the complex $(E_1,d_1)$ can be identified with $\{H^*(T_I,\Z)\}$. Also, all differentials $d_{>1}$ are torsion since the same spectral sequence over $\Q$ degenerates at $E_2$ (meaning all differentials starting from the second one are zero) by the proof of Theorem~\ref{rational}. 

So all torsion present in $E_2$ is also present in $H^*(\{H^*(T_I,\Z)\})$, and the higher differentials $d_{>1}$ cannot create any new torsion. Finally, passing from $E_\infty$ to cohomology may require non-trivial extensions, but this does not affect the cardinality of the torsion. This proves Proposition~\ref{bound_on_torsion}.$\square$

\section{Proof of Theorem~\ref{main2}}\label{more_proofs}
In this section the ground ring is $\Z$.

\subsection{\texorpdfstring{Comparing $\Bin$ with $\Gamma$}{}}

{\bf Preliminary observations.} Let $\Fin_*$ be the category of finite based sets.
Given $S\in \Fin_*$, let $\mb{Z}\langle S/*\rangle$ denote the quotient of $\mb{Z}\langle S\rangle$, the free Abelian group generated by $S$, by $\mb{Z}\cdot *$ where $*\in S$ is the base point.
The essential image of the functor 
$\mb{Z}\langle-/*\rangle   \colon \Fin_*\ar \Ab$
will be denoted by $\Ab'$: this is the subcategory of Abelian groups with 
a chosen (unordered) basis subset equal to $S\setminus\{*\}$, and with the morphisms 
$\Ab'$ being the $\Z$-linear maps which take a basis element to zero or to another basis element.

The opposite category $\Fin^{\op}_*$ admits a similar description.
There is a functor $S\ar \mb{Z}^{S/*}$ from $\Fin^{\op}_*$ to $\Ab$,
which maps a based set $S$ to the integer-valued functions on $S$ vanishing at $*\in S$.
The essential image of this functor will be denoted by $\Ab''$.
Let $\Bin''$ and $\Gamma''$ be the restrictions of the functors 
$\Bin(-)$ and $\Gamma(-)$ to $\Ab''$.
\begin{prop}\label{bin''_gamma''}
There is a natural isomorphism of functors $\Bin''\simeq \Gamma''$.
\end{prop}
\begin{proof}
Let $F$ denote either $\Bin''$ or $\Gamma''$.
The diagonal map $\mb{Z}\ar \mb{Z}\oplus \mb{Z}$ is a morphism in $\Ab''$. It
induces a coalgebra structure on the Abelian group $F(\mb{Z})$.
Similarly, the morphism $0\ar \mb{Z}$ yields the coagmentation $\mb{Z}\ar F(\mb{Z})$.

We claim that the functor $F$ is determined by the coaugmented coalgebra $F(\mb{Z})$.
Namely, there is a natural identification $F(\mb{Z}^{S/*})\simeq \otimes_{s\in S\setminus\{*\}}
F(\mb{Z}\cdot s)$ induced by the commutative algebra structure on $F(\mb{Z}^{S/*})$ and the functoriality
along the maps $(\{s,*\},*)\ar (S,*)$.
Explicitly, $\Bin(\mb{Z}^{S/*})$ is spanned by the binomial coefficients $\prod_{s\in S\setminus\{*\}}
\binom{x_s}{n_s}$, while $\Gamma(\mb{Z}^{S/*})$ is freely spanned by the divided powers
$\prod_{s\in S\setminus\{*\}}y^{[n_s]}_s$ where $x_s$ and $y_s$ correspond to the delta functions on $S$ at $s\in S\setminus\{*\}$.
A map $g\colon (S,*)\ar (S',*)$ of based sets induces 
a map $G\colon F(\mb{Z}^{S'/*})\ar F(\mb{Z}^{S/*})$ which, in terms of the 
identification above, is determined
by the coalgebra structure on $F(\mb{Z})$.
Namely, for all $s'\in S'\setminus \{*\}$ we have 
$$G(x_{s'})=\sum_{f(s)=s'}x_s,$$ 
where the summation is over the preimage of $s'$; similarly for 
$y_{s'}$.
The coalgebras $\Bin(\mb{Z})$ and $\Gamma(\mb{Z})$ are isomorphic: the comultiplication
formula for $\binom{x}{n}$ and $x^{[n]}$ is the same.
Hence the map which sends $\prod_{s\in S\setminus\{*\}}
\binom{x_s}{n_s}$ to $\prod_{s\in S\setminus\{*\}}y^{[n_s]}_s$ defines 
an isomorphism $\Bin(\mb{Z}^{S/*})\simeq \Gamma(\mb{Z}^{S/*})$
which is functorial in $(S,*)$.
In other words, the functors $\Bin''$ and $\Gamma''$ are isomorphic.
\end{proof}

\medskip

{\bf A splitting of certain binomial diagrams.} Let $\Spaces^{ft}_*$ denote the category of based simplicial sets with finitely many simplices in each degree. Here we call such objects simply \emph{spaces}.
Assume $S\colon \mathcal{C}\ar \Spaces^{ft}_*$ is a diagram of spaces such that for all $d\in \mathcal{C}$,  the space $S_d$ has the integral homology of a wedge of circles.
It is well known that the associated graded with respect to the binomial filtration on $\Bin$, provided by Lemma~\ref{lemma:bin_vs_canonical}, is isomorphic to the divided powers functor $\Gamma$ (cf.~\cite{ksz}).
\begin{prop}\label{prop:gamma_bin}
There exists a natural equivalence of diagrams 
$\Gamma(H^1(S_d)[-1])\zgqis \Bin(H^1(S_d)[-1])$ of shape~$\mathcal{C}^{\op}$.
\end{prop}
\begin{proof}
The reduced complex $\tilde{C}^*(S_d)$ has the only non trivial cohomology group in degree
$1$, which implies a natural equivalence $\tilde{C}^*(S_d)\zgqis H^1(S_d;\mb{Z})[-1]$
functorial in $d\in \mathcal{C}^{\op}$.
On the other hand, the reduced complex $\tilde{C}^*(S_d)$ can be 
viewed as a cosimplicial complex in $\Ab''$ equal to 
$\mb{Z}^{S_d/*}$.  By the previous proposition we obtain an isomorphism
$\Bin(\mb{Z}^{S_d/*})\simeq \Gamma(\mb{Z}^{S_d/*})$, which is automatically 
functorial in $d$ because the diagram arrows induce morphisms 
$\mb{Z}^{S_d/*}\ar \mb{Z}^{S_{d'}/*}$ in cosimplicial objects in $\Ab''$.
Finally, Lemma \ref{left_quillen} provides a functorial equivalence 
$F(\mb{Z}^{S_d/*})\zgqis F(H^1(S_d)[-1])$ for $F=\Bin,\Gamma$, which concludes the proof.

\end{proof}

\smallskip

We are now ready to prove Theorem~\ref{main2}. We will do this assuming a technical result (Proposition~\ref{diagrams_of_class_spaces}) which will be proved later in Section~\ref{sect:diag_class}.
Suppose $\mathcal{C}=\square_k$ is the cubical diagram corresponding to the toric arrangement $\{T_I\subset T\}$,
and $A^*$ is the corresponding 2-term complex of shape $\mathcal{C}$ described in \ref{construction_A}.

\begin{lemma}\label{lemma:decomposition}
If all $T_I$ are path connected, then 
the diagram $A^*$ is \emph{decomposable}, 
i.e.\ there is a quasi-isomorphism $A^*\zgqis H^0(A^*)\oplus H^1(A^*)[-1]$ of diagrams. 
\end{lemma}
\begin{proof}
We may assume that $0\in T_I$ for all $I$. Recall that the diagram term $A^0_I$ is the $\mb{Z}$-module of all functions constant along $T_I$ which are affine and integral with respect to the affine structure on $T$ (see equation~(\ref{eq:affine})). The evaluation at $0\in T$ splits off $H^0(A^*)\subset A^*$, 
hence provides a decomposition of $A^*$.
\end{proof}
{\bf Proof of Theorem~\ref{main2}.} 
By our assumptions, the diagram of free Abelian groups $H_1(T_I;\mb{Z})\in \Ab$
lifts to a diagram of free groups, i.e.\ there is a diagram of free groups $G_I\in \Groups^{free}$ 
provided with an isomorphism $G_I/[G_I,G_I]\simeq H_1(T_I;\mb{Z})$ functorial in $T_I$.


By Proposition~\ref{diagrams_of_class_spaces} 
there is a diagram 
$S\colon \mathcal{C}\ar \Spaces^{ft}_*$ where
$S_I$ has the homotopy type of the classifying space of $G_I$.
Then Proposition \ref{prop:gamma_bin} implies that each
$\Bin(H^1(S_I)[-1])\zgqis \Gamma(H^1(S_I)[-1])$. Note that the d\'ecalage isomorphism 
$\Gamma^k(H^1(S_I)[-1])\zgqis \Lambda^k(H^1(S_I))[-k]$
identifies $\Gamma(H^1(S_I)[-1])$ with $\MV(H^*(T_I,\Z))$.
Finally, since all $T_I$ are connected, by the above Lemma \ref{lemma:decomposition} we see that $A^*$ is decomposable.
Hence the complex $\MV(\Bin(H^1(S_I,\Z)[-1]))$ is equivalent to $\MV(H^*(T_I,\Z))$, and by
Theorem~\ref{main1} we conclude that $\MV(H^*(T_I))$ computes $H^*(T,\cup_i T_i,\Z)$.$\square$

\subsection{Diagrams of classifying spaces}\label{sect:diag_class}

Let $\mathcal{C}$ be the $k$-cube category $\square_k$. Let $G:\mathcal{C}\ar\mathrm{Groups}$ be a functor from $\mathcal{C}$ to groups such that each $G(I), I\in\mathcal{C}$ is a finitely generated free group.

A \emph{concrete polyhedron} is the geometric realisation of a simplicial complex. A map $f:K\ar L$ of concrete polyhedra is \emph{simplicial} if it is induced by a map of the underlying simplicial complexes, and \emph{piecewise linear}, or \emph{PL} if after subdividing every simplex of $K$ into finitely many affine simplices, the restriction of $f$ to any simplex $\triangle$ of $K$ is an affine map from $\triangle$ to some simplex of $L$. 

A \emph{polyhedral structure} on a space $X$ is a homeomorphism between $X$ and a concrete polyhedron. Two polyhedral structures on $X$ are \emph{equivalent} if both of the resulting maps between the corresponding concrete polyhedra are PL. We will refer to a space with a chosen equivalence class of polyhedral structures as a \emph{polyhedron}. 

The definition of a PL map extends to polyhedra in a straightforward way. A \emph{subpolyhedron} of a polyhedron $X$ is a subspace $X'\subset X$ equipped with a polyhedral structure such that the inclusion map $X'\ar X$ is PL. For example, open subsets of polyhedra are again subpolyhedra. One defines \emph{locally trivial PL fibre bundles} with fibre, base and total space being polyhedra by requiring that all local trivialisations should be PL.

A polyhedron $X$ is \emph{finite} if at least one (hence every) concrete polyhedron involved in the definition of the polyhedral structure on $X$ is finite. A polyhedron is \emph{locally finite} if every point of it has a closed neighbourhood which is a finite subpolyhedron.

\begin{prop}\label{diagrams_of_class_spaces}
There is a functor $\tilde{B}G(-)$ from $\mathcal{C}$ to pointed spaces, such that each $\tilde{B}G(I),I\in\mathcal{C}$ is a finite polyhedron homotopy equivalent to a classifying space for $G(I)$.
\end{prop}

\begin{proof}
We proceed in two steps. Namely, we will first find a lift of $G(-)$ to locally finite polyhedra, and then we will construct a subfunctor to finite polyhedra. Choose, for each $I\in\mathcal{C}$, a based tree $E(G(I))$ with a free simplicial action of $G(I)$ on it, so that each $E(G(I))/G(I)$ becomes a classifying space for $G(I)$. Note that $E(G(I))/G(I)$ inherits a polyhedral structure from $E(G(I))$.

{\bf Step 1.} Set, for $I\in \mathcal{C}$
$$E'(G(I))=\prod_{I\ar J} E(G(J))$$ where the product is taken over all arrows in $\mathcal{C}$ starting at $I$ (including the identity of $I$). We let $G(I)$ act diagonally on $E'(G(I))$. Moreover, if $I\ar J$ is an arrow in $\mathcal{C}$, then we have a map of $G(I)$-spaces $E'(G(I))\ar E'(G(J))$. So $$I\mapsto B'(G(I)):=E'(G(I))/G(I)$$ is a functor from $\mathcal{C}$ to pointed polyhedra where the base point in each $B'(G(I))$ is induced from $E'(G(I))$.

{\bf Step 2.} Every space $B'(G(I))$ is fibred over the compact classifying space $E(G(I))/G(I)$ with fibre equal the product
$$\prod_{\alpha:I\ar J, \alpha\neq\mathrm{Id}_I} E(G(J))$$ over all maps out of $I$ except the identity. So $B'(G(I))$ is not compact except when $I=\{1,\ldots,k\}$, and we need to trim it a little. We will do this using the next lemma:
\begin{lemma}\label{cofinal_subpolyhedra}
Let $F=T_1\times\cdots\times T_m$ be a product of locally finite trees. Suppose $B$ is a finite wedge of circles. Denote
the base point $0\in B$, and assume that $\pi_1(B,0)$ acts on $F$ by products of tree automorphisms. 
Let $F\ar X\overset{p}{\ar} B$ be the PL fibre bundle 
associated with the action. 

Then the finite subpolyhedra of $X$ which are homotopy equivalent to $X$ are cofinal in the set of all compact subspaces of $X$ ordered by inclusion.
\end{lemma}

\begin{proof} Choose an orientation on each circle in the wedge $B$, and let $g_1,\ldots,g_l$ be the corresponding system of generators of $\pi_1(B,0)$. Denote the fibre of $p$ over the base point $0\in B$ by $F_0$. 

Let $K\subset X$ be a compact subspace and
$T\subset F_0$ be a product of finite subtrees. Using parallel transfer along the natural flat connection on $X\overset{p}{\ar} B$ we can extend $T$ to a finite subpolyhedron $T'\subset X$ such that the fibre $T'\cap p^{-1}(t)$ for $t\in B\setminus \{0\}$ is homeomorphic to $T$, hence is contractible, and the fibre $T'\cap p^{-1}(0)$ is the union of $T$ and its images under $g_1,\ldots,g_l$. 
Enlarging $T$ if necessary one can assume that $T'$ contains $K$. Finally, one can choose a product of finite subtrees $S\subset p^{-1}(0)$ containing $T'\cap p^{-1}(0)$ and set $X'=T'\cup S$. By the construction, $X'\subset X$ is a finite subpolyhedron, and $X'\ar B$ is a proper map with contractible fibers, hence a homotopy equivalence.
\end{proof}

We continue with the proof of Proposition~\ref{diagrams_of_class_spaces}. We proceed by induction on $|I|$. For $I=\varnothing$ we let $B''(G(I))$ be any finite subpolyhedron of $B'(G(I))$ which is homotopy equivalent to $B'(G(I))$ and contains the base point. This is possible by Lemma~\ref{cofinal_subpolyhedra}.

Now suppose $0< m\leq k$, and that for each $I$ with $|I|<m$ we have constructed a finite subpolyhedron $B''(G(I))\subset B'(G(I))$ such that the inclusion is a homotopy equivalence, and any arrow $I_1\ar I_2, |I_1|\leq |I_2|<m$ takes $B''(G(I_1))$ to $B''(G(I_2))$. Let $I$ be such that $|I|=m$. Using Lemma~\ref{cofinal_subpolyhedra} again we define then $B''(G(I))\subset B'(G(I))$ to be a finite subpolyhedron such that it is homotopy equivalent to $B'(G(I))$ and contains the base point and the images of all arrows that end at $I$. This completes the induction step in the construction of the functor $B''(G(-))$, and also the proof of Proposition~\ref{diagrams_of_class_spaces}.
\end{proof}

\begin{remark}
Lemma~\ref{cofinal_subpolyhedra} can probably be extended to include more general spaces. For example, suppose $X$ is a polyhedron for which there is a finite subpolyhedron $X'\subset X$ such that the inclusion is a homotopy equivalence. Does it follow that finite polyhedra with this property are cofinal in the set of all compact subspaces of $X$? If the answer is no in general, will it be positive if one assumes in addition that $X$ is locally finite and/or finite dimensional? 
\end{remark}

\begin{remark}
Proposition~\ref{diagrams_of_class_spaces} can be generalised to $\mathcal{C}$ equal an arbitrary finite poset. 
\end{remark}

\begin{flushleft}
Alexey Gorinov: Faculty of Mathematics, Higher School of Economics, email: \url{agorinov@hse.ru}, \url{gorinov@mccme.ru}
\end{flushleft}
\begin{flushleft}
Alexander Zakharov: Chennai Mathematical Institute, email: \url{azakharov@cmi.ac.in}
\end{flushleft}

\end{document}